\documentclass[oribibl]{llncs}

\usepackage{makeidx}
\usepackage{algorithm}
\usepackage{amsmath}
\usepackage{amsfonts}
\usepackage{cite}
\usepackage{algpseudocode}
\usepackage{graphicx}
\begin{document}

%
%
\pagestyle{headings}  
%

\mainmatter              
%
\title{Explicit Bound for the Prime Ideal Theorem in Residue Classes}

%
%
%
\author{Maciej Grze\'skowiak \thanks{The author was partially supported by the grant no. 2013/11/B/ST1/02799 from
the National Science Centre.} }
%
%

\institute{Adam Mickiewicz University,\\ Faculty of Mathematics and Computer Science,\\
Umultowska 87, 61-614 Pozna\'n, Poland\\
\email{maciejg@amu.edu.pl}}

\maketitle

\begin{abstract}
 We give explicit numerical estimates for the generalized Chebyshev functions. Explicit results of this kind are useful  for estimating of computational complexity of  algorithms which generates special primes. Such primes are needed to construct an elliptic curve over prime field using complex multiplication method.
\end{abstract}

\section{Introduction}\label{S:intro}
Let $K$ denote any fixed totally imaginary field of the discriminant $\Delta=\Delta(K)$ and  degree $[K :\mathbb{Q}]=2r_2$, where $2r_2$ is the number of complex-conjugate fields of $K$.  Denote by $\mathfrak{f}$ a given non-zero integral ideal of
the ring of algebraic integers $\mathcal{O}_K$ and by $H \pmod{\mathfrak{f}}$ any ideal class mod $\mathfrak{f}$ in the ''narrow'' sense. Let $h_\mathfrak{f}^*(K)$ be the number of elements of $H$.
Let $\chi(H)$ be a character of the abelian group  of ideal classes $H \pmod{\mathfrak{f}}$, and let $\chi(\mathfrak{a})$ be the usual extension
of $\chi(H)$. Let $s=\sigma+it$. The Hecke--Landau zeta-functions associated to $\chi$, are defined by
\begin{eqnarray*}
\zeta(s,\chi)=\sum_{\mathfrak{a}\in \mathcal{O}_K}\frac{\chi(\mathfrak{a})}{(N \mathfrak{a})^s},  \qquad \sigma >1,
\end{eqnarray*}
where $\mathfrak{a}$ runs through integral ideals  and $N \mathfrak{a}$ is the norm of $\mathfrak{a}$. Throughout, $\chi_0$ denote the principal character modulo $\mathfrak{f}$.
Let
\begin{align*}
E_0=E_0(\chi)=\left\{\begin{array}{ccc}
                1 &  \mbox{for} & \chi=\chi_0 \\
                0 &  \mbox{for} & \chi\neq\chi_0
              \end{array}\right.
\end{align*}
If $\chi$ is a primitive character, then $\zeta(s,\chi)$ satisfies the functional equation
\[
\Phi(s,\chi)=W(\chi)\Phi(1-s,\overline{\chi}),\qquad |W(\chi)|=1,
\]
where
\begin{eqnarray*}
\Phi(s,\chi)=A(\mathfrak{f})^s\Gamma(s)^{r_2}\zeta(s,\chi)
\end{eqnarray*}
and
\[
A(\mathfrak{f})=(2\pi)^{-r_2}\sqrt{|\Delta | N\mathfrak{f}}.
\]
Let $\Lambda(\mathfrak{a})$ be the generalized Mangoldt function. Fix $X \mod{\mathfrak{f}} \in H$.  We define,
\begin{align*}
\Psi(x,X)=\sum_{\stackrel{\scriptstyle x \leq N\mathfrak{a}\leq 2x}{\mathfrak{a}\in X}}\Lambda(\mathfrak{a})=\sum_{\stackrel{\scriptstyle x \leq N\mathfrak{p}^m\leq 2x}{\mathfrak{p}^m\in X}}\log N \mathfrak{p},
\end{align*}
 where $\mathfrak{p}$ runs through prime ideals of $\mathcal{O}_K$. The aim of this paper is to proof the following theorem.
\begin{theorem}\label{mainTheorem}
Let $K$, $\Delta$, $\mathfrak{f}$, $\zeta(s,\chi)$ denote respectively any algebraic number field of degree $[K : \mathbb{Q}] =2r_2$, the discriminant of $K$, any ideal in $K$ and any Hecke-Landau zeta function with
a character $\chi$ modulo $\mathfrak{f}$. Fix $0< \varepsilon <1$. If $|\Delta|\geq 9$ and there is no  zero in the region
\begin{align}\label{OBszarGL}
\sigma\geq 1-0.0795\left(\log |\Delta| + 0.7761\log\left((|t|+1)^{2r_2}(N\mathfrak{f})^{1-E_0}\right)\right)^{-1},
\end{align}
then
\begin{align*}
& \Psi(x, X) \geq \frac{x(1 -\varepsilon)}{h_{\mathfrak{f}}^*(K)} ,
\end{align*}
for
\begin{align*}
\log x\geq \left(23.148\sqrt{r_2}\left(1+\left(2\log\left(\frac{c_1\sqrt{r_2}}{0.117\varepsilon}\right)\right)^{\frac{1}{2}}+\frac{2}{3}\log\left(\frac{c_1\sqrt{r_2}}{0.117 \varepsilon}\right)\right)\right)^2.
\end{align*}
where
\begin{align*}
c_1 & =(36997.123|\Delta|^{\frac{1.933}{r_2}} + 19064.499|\Delta|^{\frac{1.289}{r_2}}(N\mathfrak{f})^{\frac{1}{r_2}}h_{\mathfrak{f}}^*(K))r_2^2\log(|\Delta|N\mathfrak{f}).
\end{align*}
\end{theorem}
\begin{remark}
For real $\chi \pmod{\mathfrak{f}}$ may exist in (\ref{OBszarGL}) one zero of $\zeta(s,\chi)$, which may be real and simple. However,  we can check numerically that a Hecke-Landau $\zeta$ function has a simple real zero
in (\ref{OBszarGL}) using scripts for computing zeros zeta functions associated to characters of finite order \cite{MaciejR}.
\end{remark}

Explicit results of this kind are useful  for estimating of computational complexity of algorithms which generates special primes. Such primes  can be used in computational number theory and cryptography.
In order to calculate exactly the time of  the algorithms one need an explicit bound for the number of desired primes from the interval $[x, 2x]$, $x\geq x_0$, where $x_0$ is computed  explicitly.
We give an example of such an algorithm. For this reason  we recall the definition  \cite{grzeskowiakECC}.
\begin{definition}
Let $p, q$ be a pair of primes and $\Delta < 0$. The primes $p, q$ are defined to be  CM-primes with respect to $\Delta$ if there exist integers $f$ and $t$ such that
\begin{eqnarray}\label{warunkistartowe}
|t| \leq 2\sqrt{p}, \quad q |p+1-t, \quad 4p-t^2=\Delta f^2.
\end{eqnarray}
\end{definition}
If CM-primes $p$ and $q$ with respect to $\Delta$ and integers $f, t$ are given, then an ordinary elliptic curve $E$ over $\mathbb{F}_p$ of cardinality $p+1-t$ can be constructed using complex multiplication method \cite{atkin}, \cite{dupont}. Let $E(\mathbb{F}_p)$ be the group of points on $E$  over $\mathbb{F}_p$, and let  $|E(\mathbb{F}_p)|$ be the order of $E(\mathbb{F}_p)$.
The group $E(\mathbb{F}_p)$ can  be used to implement public key cryptographic systems, based on intractability of the discrete logarithm problem (DLP).
To make the DLP in $E(\mathbb{F}_p)$ intractable, it is
essential to generate a large prime $p$, and a curve $E$ defined over $\mathbb{F}_p$, such that $|E(\mathbb{F}_p)|$  has a large prime factor $q$.
In  \cite{grzeskowiakECC} an algorithmic method for constructing a pair $(E, p)$ such that $|E(\mathbb{F}_p)|$ has a large prime factor $q$ is given.
Fix $K$ an imaginary quadratic number field, and positive integers $m, n$, $(n,m)=1$. Then the algorithm generates $\alpha  \in \mathcal{O}_K$ such that $q=N_{K/\mathbb{Q}}(\alpha) \equiv m \pmod n$ is a prime, and
$x\leq q \leq 2x$ for sufficiently large $x \geq x_0$. Given $\alpha, q$ a prime $p$, $x < p < x^{\frac{5}{2-5\varepsilon}}$, is constructed, where  $0<\varepsilon < \frac{2}{5}$.  For more algorithms of this kind we refer the reader to \cite{Relatively}, \cite{Pairing-friendly}.

Let $x \in \mathbb{R}$, and let $W(x)$ be the Lambert $W$ function such that $W(x)e^{W(x)}=x$. If $-e^{-1} \leq x \leq 0$, then there are two possible real values of $W(x)$. We denote the branch satisfying $-1 < W(x)$ by
$W_0(x)$ and the branch satisfying $W(x) \leq -1$ by $W_{-1}(x)$.
Fix $X \pmod{ \mathfrak{f}} \in H$. We define
\begin{align*}
\psi(x, X) : =\sum_{\stackrel{\scriptstyle N\mathfrak{p}^m < x }{\mathfrak{p}^m \in X}}\log N \mathfrak{p},
\end{align*}
where $\mathfrak{p}$ runs through prime ideals of $\mathcal{O}_K$. Theorem \ref {mainTheorem} follows from Theorem \ref{oszcaowanoeS2}.
\begin{theorem}\label{oszcaowanoeS2}
Let $K$, $\Delta$, $\mathfrak{f}$, $\zeta(s,\chi)$ denote respectively any algebraic number field of degree $[K : \mathbb{Q}] =2r_2$, the discriminant of $K$, any ideal in $K$ and any Hecke-Landau zeta function with
a character $\chi$ modulo $\mathfrak{f}$. Let $A_0=0.7761$. If $|\Delta|\geq 9$ and there is no zero in the region
\begin{align*}
\sigma\geq 1-0.0795\left(\log |\Delta| + A_0\log \left((|t|+1)^{2r_2}(N\mathfrak{f})^{1-E_0(\chi)}\right)\right)^{-1},
\end{align*}
then
\begin{align*}
& \psi(x, X) \geq \frac{x}{h_{\mathfrak{f}}^*(K)} - \frac{c_2x}{h_{\mathfrak{f}}^*(K)}(\log x)^{\frac{1}{2}}e^{-0.0432r_2^{-1/2}\sqrt{ \log x }},
\end{align*}
and
\begin{align*}
& \psi(x, X)  \leq \frac{x}{h_{\mathfrak{f}}^*(K)} + \frac{c_3x}{h_{\mathfrak{f}}^*(K)}(\log x)^{\frac{1}{2}}e^{-0.0459r_2^{-1/2}\sqrt{ \log x }}
\end{align*}
for $x \geq \exp\left( 116r_2 \log \left(2|\Delta|^{\frac{1}{A_0r_2}}(N\mathfrak{f})^{\frac{1}{r_2}}\right)\right)$,
where
\begin{align*}
c_2 & =(10756.659|\Delta|^{\frac{3}{2A_0r_2}} + 5541.374|\Delta|^{\frac{1}{A_0r_2}}(N\mathfrak{f})^{\frac{1}{r_2}}h_{\mathfrak{f}}^*(K))r_2^2\log(|\Delta|N\mathfrak{f}),\\
c_3 & =(14665.542|\Delta|^{\frac{3}{2A_0r_2}} + 7555.065|\Delta|^{\frac{1}{A_0r_2}}(N\mathfrak{f})^{\frac{1}{r_2}}h_{\mathfrak{f}}^*(K))r_2^2\log(|\Delta|N\mathfrak{f}).
\end{align*}
\end{theorem}
\begin{proof}
See Section \ref{SekcjaTH2}.
\end{proof}
We are now in a position to prove Theorem \ref{mainTheorem}
\begin{proof}
By Theorem \ref{oszcaowanoeS2} we have
\begin{align*}
& \psi(2x, X)- \psi(x, X) \geq \frac{x}{h_{\mathfrak{f}}^*(K)} - \frac{c_1x}{h_{\mathfrak{f}}^*(K)}(\log x)^{\frac{1}{2}}e^{-0.0432r_2^{-1/2}\sqrt{\log x}},
\end{align*}
where
\[
c_1=\left(2c_2\left(1+\frac{\log{2}}{\log x}\right)^{\frac{1}{2}} + c_3\right)\leq 2.077c_2+c_3
\]
for $x \geq \exp\left( 116r_2 \log \left(2|\Delta|^{\frac{1}{A_0r_2}}(N\mathfrak{f})^{\frac{1}{r_2}}\right)\right)$.
Fix $0< \varepsilon < 1$. If
\begin{align*}
c_1(\log x)^{\frac{1}{2}}e^{-0.0432r_2^{-1/2}\sqrt{\log x}}\leq\varepsilon,
\end{align*}
then
\begin{align*}
0.0432r_2^{-\frac{1}{2}}\left(\log x\right)^{\frac{1}{2}}\geq - W_{-1}\left(\frac{-0.0432\varepsilon}{c_1\sqrt{r_2}}\right).
\end{align*}
By \cite[ Theorem 1]{LambertW}
\begin{align*}
\log x\geq \left(23.148\sqrt{r_2}\left(1+\left(2\log\left(\frac{c_1\sqrt{r_2}}{0.117\varepsilon}\right)\right)^{\frac{1}{2}}+\frac{2}{3}\log\left(\frac{c_1\sqrt{r_2}}{0.117 \varepsilon}\right)\right)\right)^2.
\end{align*}
This finishes the proof.
\end{proof}

\section{The proof of Theorem \ref{oszcaowanoeS2}}\label{SekcjaTH2}
The proof  of Theorem \ref{oszcaowanoeS2}  rests on the following lemmas and theorems.
\begin{theorem}\label{Fryska_obszr_wolnyodZer}
Let $K$, $\mathfrak{f}$, $\zeta(s,\chi)$ denote respectively any algebraic number field of degree $n\geq 2$, any ideal in $K$ and any Hecke-Landau zeta function with
a character $\chi$ modulo $\mathfrak{f}$. Let futher
\begin{align}\label{sta쓰L}
L(t)=\log |\Delta| + A_0\log\left((|t|+1)^n(N\mathfrak{f})^{1-E_0}\right)\geq 2.097.
\end{align}
Then in the case of the complex $\chi$ in the region
\begin{eqnarray}\label{obszar_wolny_od_zer}
\sigma\geq 1-\frac{A_1}{L(t)}\geq 1- 0.037911=0.962088=A_{2}
\end{eqnarray}
there is no zero of $\zeta(s,\chi)$, where  $A_0=0.7761$, $A_1=0.0795$. For real $\chi \pmod{\mathfrak{f}}$ they may exist in (\ref{obszar_wolny_od_zer}) one zero of $\zeta(s,\chi)$, which may be real and simple.
\end{theorem}
\begin{proof}
See \cite[Th. 2]{fryskaObszar}.
\end{proof}
\begin{lemma}\label{PochodnaLogarytmicznaZeta}
Let $s=\sigma +it$, $0< \eta \leq \frac{1}{4}$,  $A_3=75.472$, $A_4=0.010$ and $|\Delta|\geq 9$. Assume that there is no exceptional zero in the region (\ref{obszar_wolny_od_zer})
Then  in the strip
$1-\frac{A_1}{6L(t)}\leq \sigma \leq 3$ we have.
\begin{align*}
& \left|\frac{\zeta'}{\zeta}(s,\chi_0) + \frac{1}{s-1}\right|  \leq  \phi_0(t, r_2, \eta, \Delta, \mathfrak{f}),
\end{align*}
where
\begin{align}
\begin{aligned}\label{phiZero}
& \phi_0(t,r_2, \eta, \Delta, \mathfrak{f} )  = 32 \log \left(L(t)(|t|+4)(|t|+2)^{r_2(1+\eta)}\left(1+A_3L(t)\right)^{2r_2}\right)\\
& \qquad + 32 \log \left(A_3(|\Delta| N\mathfrak{f})^{\frac{1+\eta}{2}}\zeta(1+\eta)^{2r_2}\right) + 8A_3r_2L(t) + \frac{A_4r_2}{L(t)},
\end{aligned}
\end{align}
and
\begin{align*}
\left|\frac{\zeta'}{\zeta}(s,\chi)\right| & \leq \phi(t, r_2, \eta, \Delta, \mathfrak{f}),
\end{align*}
where
\begin{align}
\begin{aligned}\label{phiNotZero}
& \phi(t, r_2, \eta, \Delta, \mathfrak{f})  = 32\log\left(\left(1+A_3L(t)\right)^{2r_2}(|t|+2)^{r_2(1+2\eta)}\right)\\
& \qquad + 32\log\left(1.4 (1+\varepsilon_\chi)A(\mathfrak{f})^{1+2\eta} \zeta(1+\eta)^{2r_2}\right)+  4A_3r_2 L(t) + \frac{A_4r_2}{L(t)}
\end{aligned}
\end{align}
for any character $\chi\neq \chi_0$ modulo $\mathfrak{f}$, where $\varepsilon_\chi=0$ or 1 to accordingly whether $\chi$ is primitive or not.
\end{lemma}
\begin{proof}
See Section \ref{Dowod_PochodnaLogarytmicznaZeta}
\end{proof}
\begin{lemma}\label{PodstawianieXzaT}
Let $\phi_0$, $\phi$ be functions defined in Lemma \ref{PochodnaLogarytmicznaZeta}. Let $T \geq 1$, $w\geq 1$, $|\Delta|\geq 9$, $c_4=\frac{1}{\sqrt{2wr_2}}$ and
\begin{align}\label{StalaC0}
c_0=c_0(\Delta, \mathfrak{f}, r_2, E_0)=|\Delta|^{-\frac{1}{2A_0r_2}}(N\mathfrak{f})^{-\frac{1-E_0}{2r_2}}
\end{align}
If
\begin{align}\label{PodstawienieTplus1}
T+1=c_0\exp\left(c_4\sqrt{\log x}\right),
\end{align}
then
\begin{align*}
\phi(T, r_2,\eta,\Delta, \mathfrak{f}) & \leq   230.911  r_2^{\frac{3}{2}}\log(|\Delta|N\mathfrak{f})(\log x)^\frac{1}{2}\\
\phi_0(T, r_2,\eta,\Delta, \mathfrak{f}) & \leq  412.531 r_2^{\frac{3}{2}}\log(|\Delta|N\mathfrak{f})(\log x)^\frac{1}{2}
\end{align*}
for $x \geq \exp\left(({c_4^{-1} \log (2c_0^{-1})})^2\right)$.
\end{lemma}
\begin{proof}
Since $T+1 \geq 2$, $\log x \geq  2r_2w \left( \log (2c_0^{-1})\right)^2 \geq 2(\log( 2 \cdot 9^\frac{1}{2A_0}))^2=8.892$, and hence $x > e^{8.892}$.
By (\ref{sta쓰L}), (\ref{PodstawienieTplus1}) we obtain
\begin{align}\label{PodstawienieLdoT}
L(T)=\frac{A_0\sqrt{2r_2}}{\sqrt{w}}(\log x)^{\frac{1}{2}}, \quad L(T)\geq 2.097.
\end{align}
Let $|\Delta| \geq 9$, $x > e^{8.892}$ and $\eta=\frac{1}{4}$. We have $\zeta\left(\frac{5}{4}\right)\leq 4.596$,
\begin{align*}
32\log(1+A_3L(T))^{2r_2}& \leq 32r_2\log\log x  + 64r_2\log(A_0\sqrt{2r_2}(A_3+\frac{1}{2.097}))\\
&\leq 111.419r_2^{\frac{3}{2}}\log\log x
\end{align*}
\begin{align*}
32r_2(1+2\eta)\log(T+2) \leq \frac{32(1+2\eta)\sqrt{r_2}}{\sqrt{2}}(\log x)^{\frac{1}{2}} \leq 37.336r_2^{\frac{1}{2}}(\log x)^{\frac{1}{2}},
\end{align*}
\begin{align*}
4A_3r_2L(T)\leq 4 \sqrt{2} A_0A_3r_2^{\frac{3}{2}}(\log x)^{\frac{1}{2}} \leq 331.334r_2^{\frac{3}{2}}(\log x)^{\frac{1}{2}},
\end{align*}
\begin{align*}
& 32\log\left(1.4 (1+\varepsilon_\chi)A(\mathfrak{f})^{1+2\eta} \zeta(1+\eta)^{2r_2}\right)+ \frac{A_4r_2}{L(T)} \leq 32\log\left(2.8\zeta(1+\eta)^{2r_2}\right)\\
& \qquad +16(1+2\eta)\log(|\Delta|N \mathfrak{f}) +  \frac{A_4r_2}{2.097}\leq 73.419r_2\log(|\Delta|N\mathfrak{f}).
\end{align*}
By the above and (\ref{phiNotZero}) we obtain
\begin{align*}
\phi(T, r_2,\eta,\Delta, \mathfrak{f})  \leq 230.911  r_2^{\frac{3}{2}}\log(|\Delta|N\mathfrak{f})(\log x)^\frac{1}{2},
\end{align*}
Similarly,
\begin{align*}
32\log L(T)\leq 16 \log\log x + 32\log(\sqrt{2r_2}A_0)\leq24.686r_2^{\frac{1}{2}}\log\log x,
\end{align*}
\begin{align*}
32\log (T+4)^{r_2(1+\eta)+1}\leq56.3r_2(\log x)^{\frac{1}{2}},
\end{align*}
and
\begin{align*}
32 \log \left(A_3(|\Delta| N\mathfrak{f})^{\frac{1+\eta}{2}}\zeta(1+\eta)^{2r_2}\right)  + \frac{A_4r_2}{L(t)} \leq 84.361r_2\log(|\Delta|N\mathfrak{f}).
\end{align*}
By the above and (\ref{phiZero}) we obtain
\begin{align*}
\phi_0(T, r_2,\eta,\Delta, \mathfrak{f}) \leq 412.531  r_2^{\frac{3}{2}}\log(|\Delta|N\mathfrak{f})(\log x)^\frac{1}{2}.
\end{align*}
This finishes the proof.
\end{proof}
\begin{lemma}\label{PodstawianieXzaT2}
 Let $T \geq 1$, $w\geq 1$, $|\Delta|\geq 9$, and let $k\geq 1$. Let $c_0$ be  the constant appearing in (\ref{StalaC0}).
If
\begin{align}\label{PodstawienieTplus12}
T+1=c_0\exp\left(\sqrt{\frac{\log x}{2wr_2}}\right),
\end{align}
then
\begin{align}
\frac{1}{T^k} \leq 2^k c_0^{-k}e^{-k c_4\sqrt{\log x}} \qquad \mbox{for} \qquad  \log x \geq  (c_4^{-1}\log(2 c_0^{-1}))^2,
\end{align}
\begin{align}\label{LogTPLUS1}
\log(e(T+k)) \leq c_4\sqrt{\log x} + \log\left(e\left(\frac{k+1}{2}\right)\right).
\end{align}
\end{lemma}
\begin{proof}
By (\ref{PodstawienieTplus12}) we have
\begin{align*}
\frac{1}{T^k} = \exp(-k\log( c_0 e ^{c_4\sqrt{\log x}}(1- (c_0e^{c_4\log x})^{-1}))) \leq \exp(-k\log( \frac{1}{2}c_0 e ^{c_4\sqrt{\log x}})),
\end{align*}
for $\log x \geq  (c_4^{-1}\log(2 c_0^{-1}))^2$. The proof of (\ref{LogTPLUS1}) is left to the reader. This finishes the proof.
\end{proof}
\begin{lemma}\label{PomLemat1}  For $T\geq 1$ we have
\begin{align*}
\int\limits_{T}^{\infty}t^{-2}dt\leq  T^{-1}, \qquad \int\limits_{T}^{\infty}t^{-2}\log(t+4)dt\leq T^{-1}\log(e(T+4))
\end{align*}
\end{lemma}
\begin{proof}
The proof is left to the reader.
\end{proof}
\begin{lemma}\label{LematPom3} Let $L(t)$  be the function which occur in (\ref{sta쓰L}). For $T\geq 1$ we have
\begin{align*}
\int\limits_{T}^{\infty}t^{-2}L(t)dt\leq c_5T^{-1}\log(e(T+4)),
\end{align*}
where $c_5= 1.09 r_2\log \left(|\Delta|(N\mathfrak{f})^{A_0(1-E_0)}\right)$.
\end{lemma}
\begin{proof}
We have
\begin{align*}
& \int\limits_{T}^{\infty}t^{-2} L(t)dt\leq 2r_2A_0\int\limits_{T}^{\infty}t^{-2}\log(t+4)dt + \log \left(|\Delta|(N\mathfrak{f})^{A_0(1-E_0)}\right)\int\limits_{T}^{\infty}t^{-2}dt.
\end{align*}
The Lemma \ref{LematPom3} follows from Lemma \ref{PomLemat1}.  This finishes the proof.
\end{proof}
\begin{lemma}\label{LematPom4}  Let $L(t)$  be the function which occur in (\ref{sta쓰L}). For $T\geq 1$ we have
\begin{align*}
\int\limits_{T}^{\infty}t^{-2}\log(1+A_3L(t))^{2r_2}dt\leq c_6T^{-1}\log(e(T+4)),
\end{align*}
where   $c_6= 11.605r_2^2\log \left(|\Delta|(N\mathfrak{f})^{A_0(1-E_0)}\right)$, and $A_3$ is the constant appearing in Lemma \ref{PochodnaLogarytmicznaZeta}.
\end{lemma}
\begin{proof}
By (\ref{sta쓰L}) we have
\begin{align*}
&\int\limits_{T}^{\infty}t^{-2} \log\left(1+A_3L(t)\right)^{2r_2}dt\leq 2r_2c_7\int\limits_{T}^{\infty}t^{-2} dt +2r_2\int\limits_{T}^{\infty}t^{-2}  L(t) dt,
\end{align*}
where $c_7=\log\left (A_3\left(1+\frac{1}{2.097A_3}\right)\right)$. The Lemma \ref{LematPom4} follows from Lemma \ref{PomLemat1} and Lemma \ref{LematPom3}.  This finishes the proof.
\end{proof}
\begin{lemma}\label{oszacowanie_calki_fi_0}  Let $\phi_0$  be the function which occur in (\ref{phiZero}).  For $T\geq 1$ we have
\begin{align*}
\int\limits_{T}^{\infty}\phi_0(t, r_2, \eta, \Delta, \mathfrak{f})t^{-2}dt\leq c_8T^{-1}\log(e(T+4)),
\end{align*}
where  $c_8  = 1138.428r_2^2\log(|\Delta|(N\mathfrak{f})^{\frac{5}{8}})$.
\end{lemma}
\begin{proof}
By (\ref{sta쓰L}) and Lemma \ref{PochodnaLogarytmicznaZeta}  with $\eta=\frac{1}{4}$ we have
\begin{align*}
& \int\limits_{T}^{\infty}\phi_0(t, r_2, \eta, \Delta, \mathfrak{f})t^{-2}dt\leq  (40r_2+32)\int\limits_{T}^{\infty}t^{-2} \log(t+4)dt\\
 &  \qquad +\left(32 \log (A_3  (|\Delta| N\mathfrak{f})^{\frac{5}{8}})+64r_2\log\zeta(\frac{5}{4})+\frac{A_4r_2}{2.097}\right)\int\limits_{T}^{\infty}t^{-2}dt\\
&  \qquad + (32+8A_3r_2)\int\limits_{T}^{\infty}t^{-2}L(t)dt +   32\int\limits_{T}^{\infty}t^{-2} \log \left(1+A_3L(t)\right)^{2r_2} dt
\end{align*}
The Lemma \ref{oszacowanie_calki_fi_0} follows from Lemmas \ref{PomLemat1},  \ref{LematPom3} and \ref{LematPom4}.
This finishes the proof.
\end{proof}
\begin{lemma}\label{oszacowanie_calki_fi}
Let $\phi$  be the function which occur in (\ref{phiNotZero}). For $T\geq 1$ we have
\begin{align*}
\int\limits_{T}^{\infty}\phi(t, r_2, \eta, \Delta, \mathfrak{f})t^{-2}dt\leq c_9T^{-1}\log(e(T+4)),
\end{align*}
where  $c_9= 821.212r_2^2\log \left(|\Delta|N\mathfrak{f}\right)$.
\end{lemma}
\begin{proof}
By (\ref{sta쓰L}) and Lemma \ref{PochodnaLogarytmicznaZeta}  with $\eta=\frac{1}{4}$ we have
\begin{align*}
& \int\limits_{T}^{\infty}\phi(t, r_2, \eta, \Delta, \mathfrak{f})t^{-2}dt \leq \left(32\log\left(2.8A(\mathfrak{f})^{\frac{3}{2}} \zeta(\frac{5}{4})^{2r_2}\right)+ \frac{A_4r_2}{2.097}\right)\int\limits_{T}^{\infty}t^{-2}dt\\
& \qquad + 32\int\limits_{T}^{\infty}t^{-2} \log\left(1+A_3L(t)\right)^{2r_2}dt+ 16r_2\int\limits_{T}^{\infty}t^{-2} \log\left(t+4\right)dt\\
& \qquad + 4A_3r_2\int\limits_{T}^{\infty}L(t) t^{-2}dt.
\end{align*}
The Lemma \ref{oszacowanie_calki_fi} follows from Lemmas \ref{PomLemat1}, \ref{LematPom3} and \ref{LematPom4}.
This finishes the proof.
\end{proof}
We are now in a position to prove Theorem \ref{oszcaowanoeS2}.
\begin{proof}
Fix $T\geq 1$, and let $c=1+\frac{1}{\log x}$.
Fix $X \pmod{\mathfrak{f}}$. We define
\begin{align}\label{psi1def}
\psi_1(x, X) : =\int_0^x \psi(t, X) dt,
\end{align}
and
\begin{align*}
\gamma(n)=\sum_{\stackrel{\scriptstyle  N\mathfrak{p}^m = n}{\mathfrak{p}^m\in X}}\log N\mathfrak{p}.
\end{align*}
Hence,
\begin{align*}
\psi(x, X)=\sum_{n\leq x}\gamma(n)
\end{align*}
By partial summation we obtain
\begin{align*}
\sum_{n \leq x}(x-n)\gamma(n) = \int_0^x \psi(t, X) dt.
\end{align*}
Now, we write
\begin{align*}
f(s,\chi)=\frac{x^{s-1}}{s(s+1)}\left[-\frac{\zeta'}{\zeta}(s,\chi)\right].
\end{align*}
By Theorem B \cite[see p. 31]{Ingham} and the orthogonality properties of $\chi \pmod{\mathfrak{f}}$ we deduce the formula
\begin{align}\label{SumaPoN}
\sum_{n \leq x}(x-n)\gamma(n) = \frac{x^2}{2\pi i h_{\mathfrak{f}}^*(K)}\sum_{\chi}\overline{\chi}(X)\int\limits_{c- i \infty }^{c+i \infty }f(s,\chi)ds,
\end{align}
where $c >1$. Let $A_1$ be the constat appearing in (\ref{obszar_wolny_od_zer}), and let $B=\frac{A_1}{6}=0.0133$.
We define the contour $\mathcal{C}$ consisting of the following parts:
\begin{align}\label{konturC}
  \mathcal{C}_1 & :  s=c+it, \mbox{ where } -T\leq t \leq T,\\\nonumber
  \mathcal{C}_2 & :  s=\sigma+iT, \mbox{ where } 1-\frac{B}{L(T)}\leq \sigma \leq c,\\
  \mathcal{C}_3 &  : s=1-\frac{B}{L(t)}+it, \mbox{ where } -T \leq t \leq T.\nonumber
\end{align}
and of $\mathcal{C}_2'$ situated symmetrically  to $\mathcal{C}_2$.
If $\chi=\chi_0$, them $\frac{\zeta'}{\zeta}(s,\chi)$ has a first order pole of residue $-1$ at $s=1$. From the Cauchy formula we get
\begin{align}\label{S1calkapokonturze}
 \frac{1}{2\pi i}\int\limits_{\mathcal{C}_1}f(s,\chi)ds = \frac{\delta(\chi)}{2} -  \frac{1}{2\pi i}\int\limits_{\mathcal{C}_2+\mathcal{C}_3+\mathcal{C}_2'}f(s,\chi)ds,
\end{align}
where
    \begin{eqnarray*}
\delta(\chi)=\left\{\begin{array}{ccl}
                     1 & \mbox{ if } & \chi=\chi_0, \\
                      0 & \mbox{ if } & \chi\neq\chi_0.
                   \end{array}\right.
\end{eqnarray*}
From  (\ref{psi1def}), (\ref{SumaPoN}) and (\ref{S1calkapokonturze}) we obtain
\begin{align}\label{Psi1WarBez}
& \left|\psi_1(x, X) - \frac{x^2}{2 h_{\mathfrak{f}}^*(K)}\right| \leq \frac{x^2(I_1+I_2+I_3)}{ h_{\mathfrak{f}}^*(K)}  +\frac{x^2(J_1+J_2+J_3)}{h_{\mathfrak{f}}^*(K)},
\end{align}
where
\begin{align*}
I_1+I_2+I_3 &=\left |\frac{1}{2\pi i}\int\limits_{c-i\infty}^{c-iT}f(s,\chi_0)\right| + \left|\frac{1}{2\pi i}\int\limits_{\mathcal{C}_2+\mathcal{C}_3+\mathcal{C}_2'}f(s,\chi_0)ds\right| + \left| \frac{1}{2\pi i}\int\limits_{c+iT}^{c+i\infty} f(s,\chi_0)ds\right|,\\
J_1+J_2+J_3 & = \left |\sum_{\chi\neq\chi_0}\overline{\chi}(X)\frac{1}{2\pi i}\int\limits_{c-i\infty}^{c-iT}f(s,\chi)ds\right| + \left|\sum_{\chi\neq\chi_0}\overline{\chi}(X)\frac{1}{2\pi i}\int\limits_{\mathcal{C}_2+\mathcal{C}_3+\mathcal{C}_2'}f(s,\chi)ds \right|\\
& +\left|\sum_{\chi\neq\chi_0}\overline{\chi}(X)\frac{1}{2\pi i} \int\limits_{c+iT}^{c+i\infty} f(s,\chi)ds\right|.
\end{align*}
We define
\begin{align}\label{funkcjeH}
h_0(s,\chi_0) =\left[-\frac{\zeta'}{\zeta}(s,\chi_0)-\frac{1}{s-1}\right]\frac{x^{s-1}}{s(s+1)},\quad h_1(s)  =\frac{x^{s-1}}{s(s+1)(s-1)}.
\end{align}
We estimate the above integrals. Let  $T\geq 1$, $x\geq e^{8.892}$, $1<c\leq 1+\frac{1}{\log x}\leq 1.12$. We write
We need to consider the following  cases:\\
1. Bound over $\mathcal{C}_2$ and $\mathcal{C}_2'$, case $\chi=\chi_0$. From Lemmas \ref{PodstawianieXzaT} and \ref{PodstawianieXzaT2}   we obtain
\begin{align*}
 & \left|\frac{1}{2\pi i}\int\limits_{\mathcal{C}_2 } f(\sigma+iT,\chi_0)d\sigma\right| \leq \frac{e}{2\pi T^2 \log x } \phi_0(T,r_2,\eta,\Delta, \mathfrak{f}) + \frac{e}{2\pi T^3 \log x }\\
& \qquad \leq c_0^{-3}r_2^{\frac{3}{2}}\log(|\Delta|N\mathfrak{f})(\log x)^{-\frac{1}{2}}\left(\frac{412.531 c_0 e}{\pi} + \frac{4c_0e}{(\log x)^{\frac{1}{2}}}\right)e^{-2c_4\sqrt{\log x}}\\
& \qquad \leq c_{10}(\log x)^{-\frac{1}{2}}e^{-2c_4\sqrt{\log x}},
\end{align*}
where $c_{10}=360.992|\Delta|^{\frac{3}{2A_0r_2}}r_2^{\frac{3}{2}}\log(|\Delta|N\mathfrak{f})$.
The same bound holds with $\int_{\mathcal{C}_2'}$ in place of $\int_{\mathcal{C}_2}$.\\
2. Bound over $\mathcal{C}_3$, case $\chi=\chi_0$. Lemmas \ref{PodstawianieXzaT} and \ref{PodstawianieXzaT2} shows that
\begin{align*}
& \left|\frac{1}{2\pi i}\int\limits_{\mathcal{C}_3 } h_0\left(1-\frac{B}{L(T)}+it,\chi_0\right)dt\right| \leq \frac{1}{\pi}x^{-\frac{B}{L(T)}}\phi_0(T, r_2,\eta,\Delta, \mathfrak{f})\int\limits_{0}^{T}\frac{dt}{\left(1-\frac{B}{L(T)}\right)^2 + t^2}\\
& \qquad \leq \frac{1}{\pi}2.01e^{-c_{18}\sqrt{\log x}}\phi_0(T, r_2,\eta,\Delta, \mathfrak{f})\leq 263.939 r_2^{\frac{3}{2}}\log(|\Delta|N\mathfrak{f})(\log x)^{\frac{1}{2}}e^{-c_{18}\sqrt{\log x}}
\end{align*}
where $c_{18}= \frac{B\sqrt{w}}{A_0\sqrt{2r_2}}$. Indeed, $1-\frac{B}{L(T)}\geq 1-\frac{0.0133}{2.097}=0.993$, and
\begin{align*}
& \int\limits_{0}^{T}\frac{dt}{\left(1-\frac{B}{L(T)}\right)^2 + t^2}=\int\limits_{0}^{1}\frac{dt}{\left(1-\frac{B}{L(T)}\right)^2 + t^2} + \int\limits_{1}^{T}\frac{dt}{\left(1-\frac{B}{L(T)}\right)^2 + t^2}\\
& \qquad \leq \int\limits_{0}^{1}\frac{dt}{(0.993)^2} + \int\limits_{1}^{T}\frac{dt}{t^2} \leq\frac{1}{(0.993)^2}+1 =2.01.
\end{align*}
 Moreover,
\begin{align*}
& \left|\frac{1}{2\pi i}\int\limits_{\mathcal{C}_3 } h_1\left(1-\frac{B}{L(t)}+it\right)ds\right| \leq \frac{1}{\pi}x^{-\frac{B}{L(T)}}\int\limits_{0}^{T}\frac{dt}{\left|1-\frac{B}{L(T)} +t\right|\left|2-\frac{B}{L(T)} +t\right|\left|-\frac{B}{L(T)} +t\right|}\\
& \qquad  \leq \frac{c_{11}}{\pi} e^{-c_{18}\sqrt{\log x}}.
 \end{align*}
where $c_{11}=\frac{1}{0.993|0.993 -1|(0.993 +1) }+1\leq 73.185$.
By the above and (\ref{funkcjeH}),
\begin{align*}
& \left|\frac{1}{2\pi i}\int\limits_{\mathcal{C}_3 } f\left(1-\frac{B}{L(T)}+it,\chi_0\right)dt\right| \leq c_{12}(\log x)^{\frac{1}{2}}  e^{-c_{18}\sqrt{\log x}},
\end{align*}
where $c_{12}=267.495r_2^{\frac{3}{2}}\log(|\Delta|N\mathfrak{f})$.
Hence,
\begin{align}\label{CalkaI2}
\begin{aligned}
I_2 & \leq |\Delta|^{\frac{3}{2A_0r_2}}r_2^{\frac{3}{2}}\log(|\Delta|N\mathfrak{f})(\log x)^{\frac{1}{2}}  e^{-c_{18}\sqrt{\log x}}\\
& \cdot \left(267.495 + \frac{2\cdot 360.992 }{\log x}e^{-(2c_4-c_{18})\sqrt{\log x}}\right) \leq c_{13}(\log x)^{\frac{1}{2}}  e^{-c_{18}\sqrt{\log x}},
\end{aligned}
\end{align}
if $w < \frac{2A_0}{B}=117.14$, where $c_{13}=348.69|\Delta|^{\frac{3}{2A_0r_2}}r_2^{\frac{3}{2}}\log(|\Delta|N\mathfrak{f})$.\\
3. Bound over $\mathcal{C}_2$ and $\mathcal{C}_2'$, case $\chi\neq\chi_0$. From Lemmas \ref{PodstawianieXzaT} and \ref{PodstawianieXzaT2}  we obtain
\begin{align*}
 & \left|\frac{1}{2\pi i}\int\limits_{\mathcal{C}_2 } f(\sigma+iT,\chi_0)ds\right| \leq  c_{14}(\log x)^{-\frac{1}{2}}e^{-2c_4\sqrt{\log x}},
\end{align*}
where $c_{14}=399.594|\Delta|^{\frac{1}{A_0r_2}}(N\mathfrak{f})^{\frac{1}{r_2}}r_2^{\frac{3}{2}}\log(|\Delta|N\mathfrak{f})$.
The same bound holds with $\int_{\mathcal{C}_2'}$ in place of $\int_{\mathcal{C}_2}$.\\
4. Bound over $\mathcal{C}_3$, case $\chi\neq\chi_0$. Lemmas \ref{PodstawianieXzaT} and \ref{PodstawianieXzaT2}  shows that
\begin{align*}
& \left|\frac{1}{2\pi i}\int\limits_{\mathcal{C}_3 } f\left(1-\frac{B}{L(T)}+it,\chi_0\right)ds\right| \leq\frac{1}{\pi} x^{-\frac{B}{L(T)}}\phi(T, r_2,\eta,\Delta, \mathfrak{f})\int\limits_{0}^{T}\frac{dt}{\left(1-\frac{B}{L(T)}\right)^2 + t^2}\\
& \qquad \leq \frac{1}{\pi}2.01e^{-c_{18}\sqrt{\log x}}\phi(T, r_2,\eta,\Delta, \mathfrak{f})\leq c_{15} (\log x)^{\frac{1}{2}}e^{-c_{18}\sqrt{\log x}},
\end{align*}
where  $c_{15}=147.738r_2^{\frac{3}{2}}\log(|\Delta|N\mathfrak{f})$.
Hence, by the above
\begin{align}\label{CalkaJ2}
\begin{aligned}
J_2 & \leq 2c_{14}(\log x)^{\frac{1}{2}}e^{-2c_4\sqrt{\log x}} + c_{15} (\log x)^{\frac{1}{2}}e^{-c_{18}\sqrt{\log x}}\\
& \leq c_{16}(\log x)^{\frac{1}{2}}  e^{-c_{18}\sqrt{\log x}},
\end{aligned}
\end{align}
if $w < \frac{2A_0}{B}=117.14$, where $c_{16}=237.616|\Delta|^{\frac{1}{A_0r_2}}(N\mathfrak{f})^{\frac{1}{r_2}}r_2^{\frac{3}{2}}\log(|\Delta|N\mathfrak{f})$.
5. Bound for $\int^{c+ i \infty}_{c+i T}$, case $\chi=\chi_0$.
By (\ref{funkcjeH}) and Lemmas \ref{PodstawianieXzaT2}, \ref{oszacowanie_calki_fi_0} we obtain
\begin{align*}
& \left|\frac{1}{2\pi i}\int\limits_{c+iT}^{c+i\infty}f(s,\chi_0)ds\right| \leq \left|\frac{1}{2\pi i}\int\limits_{c+iT}^{c+i\infty}h_0(s,\chi_0)ds\right|+\left|\frac{1}{2\pi i}\int\limits_{c+iT}^{c+i\infty}h_1(s)ds\right|\\
& \qquad \leq  \frac{e}{2\pi}\int\limits_{T}^{\infty}\phi_0(t, r_2, \eta, \Delta, \mathfrak{f})t^{-2}dt+ \frac{e}{2\pi}\int\limits_{T}^{\infty}t^{-3}dt \leq \frac{e}{2\pi}c_8\frac{\log (e(T+4))}{T}\\
& \qquad +\frac{e}{4\pi T^2}\leq \frac{ec_8}{\pi c_0^2}(\log x)^{\frac{1}{2}}\left( \frac{c_0}{\sqrt{2w}} + \frac{1.917c_0}{(\log x)^{\frac{1}{2}}} +\frac{c_0}{2c_8(\log x)^{\frac{1}{2}}}\right)e^{-c_4\sqrt{\log x}}\\
& \qquad \leq c_{17}(\log x)^{\frac{1}{2}}e^{-c_4\sqrt{\log x}}
\end{align*}
for $\log x \geq  (c_4^{-1}\log(2 c_0^{-1}))^2 \geq 8.892$,  with $(\log x)^{\frac{1}{2}}\geq 2.98$, $c_0\leq 1$, $w=58$.
where $c_{17}=724.845|\Delta|^{\frac{1}{A_0r_2}}r_2^2\log(|\Delta|(N\mathfrak{f})^{\frac{5}{8}})$.
The same bound holds with $\int_{c-i\infty}^{c-iT}$ in place of $\int^{c+i\infty}_{c+iT}$.
Hence,
\begin{align}\label{CalkiI1plusI3}
I_1+I_3\leq  2c_{17}(\log x)^{\frac{1}{2}}e^{-c_4\sqrt{\log x}}.
\end{align}
6. Bound for $\int^{c+ i \infty}_{c+i T}$, case $\chi\neq\chi_0$. Lemmas \ref{PodstawianieXzaT2} and \ref{oszacowanie_calki_fi} shows that
\begin{align*}
& \left|\frac{1}{2\pi i}\int\limits_{c+iT}^{c+i\infty}f(s,\chi)ds\right|\leq  \frac{e}{2\pi} \int\limits_{T}^{\infty}\phi(t, r_2, \eta, \Delta, \mathfrak{f})t^{-2}dt
 \leq \frac{e}{2\pi}c_9\frac{\log (e(T+4))}{T}\\
 & \qquad \leq \frac{ec_9}{\pi c_0}(\log x)^{\frac{1}{2}}\left( \frac{1}{\sqrt{2w}} + \frac{1.917}{(\log x)^{\frac{1}{2}}}\right)e^{-c_4\sqrt{\log x}}   \leq c_{19}(\log x)^{\frac{1}{2}}e^{-c_4\sqrt{\log x}}.
\end{align*}
where $c_{19}=522.77|\Delta|^{\frac{1}{2A_0r_2}}(N\mathfrak{f})^{\frac{1}{2r_2}} r_2^2\log(|\Delta|N\mathfrak{f})$, and $w=58$. The same bound holds with $\int_{c-i\infty}^{c-iT}$ in place of $\int^{c+i\infty}_{c+iT}$.
Hence,
\begin{align}\label{CalkiJ1plusJ3}
J_1+J_3\leq  2c_{19}(\log x)^{\frac{1}{2}}e^{-c_4\sqrt{\log x}}.
\end{align}
By  (\ref{CalkaI2}), (\ref{CalkiI1plusI3}) we have
\begin{align}\label{CalkiI1I2I3}
& I_1+I_2+I_3\leq c_{20}(\log x)^{\frac{1}{2}}  e^{-c_{18}\sqrt{\log x}},
\end{align}
where $c_{20}=3585.536|\Delta|^{\frac{3}{2A_0r_2}}r_2^2\log(|\Delta|N\mathfrak{f})$, for $1\leq w < \frac{A_0}{B}=58.57$. From  (\ref{CalkaJ2}), (\ref{CalkiJ1plusJ3}) we obtain
\begin{align}\label{CalkiJ1J2J3}
& J_1+J_2+J_3 \leq c_{21}(\log x)^{\frac{1}{2}}  e^{-c_{18}\sqrt{\log x}},
\end{align}
where $c_{21} =1847.116h^*_{\mathfrak{f}}(K)|\Delta|^{\frac{1}{A_0r_2}}(N\mathfrak{f})^{\frac{1}{r_2}}r_2^2\log(|\Delta|N\mathfrak{f})$ for $1\leq w < \frac{A_0}{B}=58.57$.  Now, by (\ref{Psi1WarBez}), (\ref{CalkiI1I2I3}), (\ref{CalkiJ1J2J3})  we obtain
\begin{align*}
& \left|\psi_1(x, X) - \frac{x^2}{2 h_{\mathfrak{f}}^*(K)}\right| \leq  \frac{x^2}{h_{\mathfrak{f}}^*(K)}c_{22}(\log x)^{\frac{1}{2}}  e^{-c_{18}\sqrt{\log x}}
\end{align*}
where $c_{22}=c_{20} + c_{21}$. Now, let $x>2$, and  $h$ be a function of $x$ satisfying $0<h<\frac{1}{2}x$. Let $W(x)=c_{22}(\log x)^{\frac{1}{2}}  e^{-c_{18}\sqrt{\log x}}$. Since $\psi(t, X)$ is an  increasing function
\begin{align*}
& \psi(x, X)  \geq \frac{1}{h} \int^x_{x -h} \psi(t, X) dt = \frac{\psi_1(x, X)-\psi_1(x-h, X)}{h}\\
& \qquad \geq  \frac{x}{h_{\mathfrak{f}}^*(K)} - \frac{x^2}{hh_{\mathfrak{f}}^*(K)}W(x)-\frac{h}{2h_{\mathfrak{f}}^*(K)}-\frac{x^2+h^2}{h h_{\mathfrak{f}}^*(K)}W(x-h).
\end{align*}
Taking $h=xe^{-\frac{1}{2}c_{18}\sqrt{\log x}}$ and $x> (\frac{2\log 2}{c_{18}})^2$, we get
\begin{align}\label{psidoSlow}
\begin{aligned}
& \psi(x, X)  \geq  \frac{x}{h_{\mathfrak{f}}^*(K)} - \frac{x}{h_{\mathfrak{f}}^*(K)}c_{22}(\log x)^{\frac{1}{2}}e^{-\frac{1}{2}c_{18}\sqrt{\log x}} - \frac{1}{2h_{\mathfrak{f}}^*(K)}xe^{-\frac{1}{2}c_{18}\sqrt{\log x}}\\
& \qquad - \frac{x}{h_{\mathfrak{f}}^*(K)}c_{22}(\log x)^{\frac{1}{2}}e^{-c_{18}(c_{23} -0.5)\sqrt{\log x}}+ \frac{x}{h_{\mathfrak{f}}^*(K)}c_{22}(\log x)^{\frac{1}{2}}e^{-c_{18}(c_{23} +0.5)\sqrt{\log x}}\\
& \qquad \geq \frac{x}{h_{\mathfrak{f}}^*(K)} - \frac{x}{h_{\mathfrak{f}}^*(K)}c_{22}(\log x)^{\frac{1}{2}}e^{-0.47c_{18}\sqrt{\log x}}(3 + c_{24})\\
& \qquad \geq \frac{x}{h_{\mathfrak{f}}^*(K)} - \frac{c_2x}{h_{\mathfrak{f}}^*(K)}(\log x)^{\frac{1}{2}}e^{-0.47c_{18}\sqrt{\log x}},
\end{aligned}
\end{align}
where $c_{23} = (1- \frac{\log 2}{\log x}) ^{\frac{1}{2}}$, $ 0.97\leq c_{23}\leq 0.98 $, $c_{24}=\frac{1}{2c_{22}}(\log x)^{-\frac{1}{2}}\leq 0.0001$, and $c_2=c_{22}(3 + c_{24})$.
On the other hand,
\begin{align*}
& \psi(x, X)  \leq\frac{1}{h} \int_x^{x+h} \psi(t, X) dt = \frac{\psi_1(x+h, X)-\psi_1(x, X)}{h}\\
& \qquad \leq  \frac{x}{h_{\mathfrak{f}}^*(K)} + \frac{h}{2h_{\mathfrak{f}}^*(K)} +\frac{(x+h)^2}{hh_{\mathfrak{f}}^*(K)}W(x+h)+\frac{x^2}{hh_{\mathfrak{f}}^*(K)}W(x)\\
& \qquad \leq \frac{x}{h_{\mathfrak{f}}^*(K)} + \frac{x}{h_{\mathfrak{f}}^*(K)}c_{22}(\log x)^{\frac{1}{2}}e^{-\frac{1}{2}c_{18}\sqrt{\log x}}(c_{25} +5c_1)\\
& \qquad \leq \frac{x}{h_{\mathfrak{f}}^*(K)} + \frac{c_3x}{h_{\mathfrak{f}}^*(K)}(\log x)^{\frac{1}{2}}e^{-\frac{1}{2}c_{18}\sqrt{\log x}}
\end{align*}
where $c_{25}=\frac{1}{2c_{22}c_{26}}(\log x)^{-\frac{1}{2}} \leq 0.001$, $c_{26}=\left(1+\frac{\log \frac{3}{2}}{\log x}\right)^{\frac{1}{2}}\leq 1.013$,  $c_3=c_{22}(c_{25} +5c_{26})$.
Putting $c_{18}=\frac{B\sqrt{58}}{A_0\sqrt{2r_2}}=0.0919\sqrt{r_2}$ we obtain the result. This finishes the proof.
\end{proof}

\section{Proof of Lemma \ref{PochodnaLogarytmicznaZeta} }\label{Dowod_PochodnaLogarytmicznaZeta}
The proof of Lemma \ref{PochodnaLogarytmicznaZeta} rest on the following lemmas.
\begin{lemma}\label{Lematzeta_notprincipal}
Let $[K: \mathbb{Q}]=2r_2$ and $0< \eta \leq\frac{1}{4}$.  In the region $-\eta \leq \sigma \leq 3$ we have the estimate
\begin{align*}
|\zeta(\sigma+it,\chi)| \leq 1.4^{r_2}(1+\varepsilon_\chi)A(\mathfrak{f})^{1+2\eta} \zeta(1+\eta)^{2r_2} (|t|+1)^{r_2(1+2\eta)}
\end{align*}
for any character $\chi\neq \chi_0$ modulo $\mathfrak{f}$, where $\varepsilon_\chi=0$ or 1 to accordingly whether $\chi$ is primitive or not.
\end{lemma}
\begin{proof}
Consider
\begin{align*}
g(s,\chi)=\frac{\zeta(s,\chi)}{\zeta(1-s,\overline{\chi})},
\end{align*}
where  $\chi$ is a primitive character$\mod \mathfrak{f}$. From the functional equation for $\zeta(s,\chi)$ it follows that
\begin{align}\label{zeta_notprincipal0}
g(s,\chi)=W(\chi)A(\mathfrak{f})^{1-2s}\left(\frac{\Gamma(1-s)}{\Gamma(s)}\right)^{r_2}.
\end{align}
We estimate $g(s,\chi)$ on the line $s=-\eta + it$, $0\leq \eta\leq \frac{1}{4}$ using the following inequality (see \cite{fryska}, p. 58)
\begin{align}\label{zeta_notprincipal1}
\left|\frac{\Gamma(1-s)}{\Gamma(s)}\right|\leq 1.4\max(1, |s|^{1+2\eta})
\end{align}
From (\ref{zeta_notprincipal0}) and (\ref{zeta_notprincipal1}) we obtain
\begin{align}\label{zeta_notprincipal2}
|g(-\eta+it,\chi)|\leq 1.4^{r_2} \,A(\mathfrak{f})^{1+2\eta}(\max(1, |-\eta+it|^{(1+2\eta)}))^{r_2}
\end{align}
for $-\infty < t < \infty$. Write
\begin{align}\label{zeta_notprincipal2A}
G(s,\chi)=\frac{\zeta(s,\chi)}{(s+1)^{r_2(1+2\eta)}}.
\end{align}
From (\ref{zeta_notprincipal2}) we have
\begin{align}
|G(-\eta +it,\chi)|& \leq 1.4^{r_2} A(\mathfrak{f})^{1+2\eta}|\zeta(1+\eta,\overline{\chi})|\\
& \leq 1.4^{r_2} A(\mathfrak{f})^{1+2\eta}\zeta(1+\eta)^{2r_2} \nonumber
\end{align}
for $\chi\neq \chi_0$. If $\chi$ is not a primitive character, then there is an ideal $\mathfrak{f}_0$ which divides $\mathfrak{f}$, and there is a primitive character
$\psi \pmod{ \mathfrak{f}_0}$ such that
\begin{align*}
\zeta(s,\chi)=\zeta(s,\psi)\prod_{\mathfrak{p} | \mathfrak{f}, \mathfrak{p} | \mathfrak{f}_0}\left(1-\frac{\psi(\mathfrak{p})}{(N\mathfrak{p})^{s}}\right).
\end{align*}
Write $\mathfrak{f}=\mathfrak{f}_0\mathfrak{f}_1$. From \cite[see p. 60]{fryska} we get
\begin{align*}
\left|\prod_{\mathfrak{p} | \mathfrak{f}, \mathfrak{p} | \mathfrak{f}_0}\left(1-\frac{\psi(\mathfrak{p})}{(N\mathfrak{p})^{s}}\right)\right|\leq 2(N\mathfrak{f}_1)^{\frac{1}{2}+\eta}.
\end{align*}
Hence,
\begin{align}\label{zeta_notprincipal3}
|G(-\eta +it,\chi)|& \leq 1.4^{r_2} (1+\varepsilon_\chi)A(\mathfrak{f})^{1+2\eta}|\zeta(1+\eta,\overline{\chi})| \\
& \leq 1.4^{r_2} \,A(\mathfrak{f})^{1+2\eta}\zeta(1+\eta)^{2r_2} \nonumber
\end{align}
for any character $\chi\neq \chi_0$ modulo $\mathfrak{f}$, where $\varepsilon_\chi=0$ or 1 to accordingly whether $\chi$ is primitive or not.
On the other hand,
\begin{align}\label{zeta_notprincipal4}
|G(3+it,\chi)|\leq \frac{|\zeta(3+\eta,\overline{\chi})|}{(4+it)^{r_2(1+2\eta)}} \leq \frac{1}{4^{r_2}}\zeta(3+\eta)^{2r_2}.
\end{align}
Using the estimate
\begin{eqnarray*}
|\zeta(s,\chi)|\leq A_1 e^{A_2|t|},
\end{eqnarray*}
 which is valid in the strip $-\eta \leq \sigma \leq 3$, where $A_1, A_2$ depends on $K$, $\chi$, and $\mathfrak{f}$, we get
\begin{align}\label{zeta_notprincipal5}
|G(s,\chi)|=O(e^{A_3|t|})
\end{align}
for $-\eta \leq \sigma \leq 3$. From (\ref{zeta_notprincipal3})-(\ref{zeta_notprincipal5}) and the well-known theorem of Phragmen-Lindel\"{o}f we obtain
\begin{align}\label{zeta_notprincipal6}
|G(s,\chi)|\leq\frac{1.4^{r_2} (1+\varepsilon_\chi)A(\mathfrak{f})^{1+2\eta}\zeta(1+\eta)^{2r_2}}{(|t|+1)^{r_2(1+2\eta)}}
\end{align}
in the strip $-\eta \leq \sigma \leq 3$. From (\ref{zeta_notprincipal6}), (\ref{zeta_notprincipal2A})
\begin{align*}
|\zeta(s,\chi)| \leq 1.4^{r_2} (1+\varepsilon_\chi) A(\mathfrak{f})^{1+2\eta} \zeta(1+\eta)^{2r_2} (|t|+1)^{r_2(1+2\eta)}
\end{align*}
for  any character $\chi\neq \chi_0$ modulo $\mathfrak{f}$, where $\varepsilon_\chi=0$ or 1 to accordingly whether $\chi$ is primitive or not. This finishes the proof.
\end{proof}
\begin{lemma}\label{Lematodwrotnosc_zety}
For $\sigma > 1$ we have
\begin{eqnarray*}
\frac{1}{|\zeta(\sigma+it,\chi)|} \leq  \zeta_K(\sigma).
\end{eqnarray*}
\end{lemma}
\begin{proof}
See \cite[Lemma 2.4]{grzeskowiak1}.
\end{proof}
\begin{lemma}\label{OszacZetaGlowny}
Let $[K: \mathbb{Q}]=2r_2$ and $0< \eta \leq\frac{1}{4}$. In the region $-\eta \leq \sigma \leq 1+\eta$, $-\infty < t < \infty$ we have estimate
\begin{align*}
|(s-1)\zeta(s,\chi_0)|\leq (3+|t|)(1+|t|)^{r_2(1+\eta-\sigma)} (|\Delta| N\mathfrak{f})^{\frac{1+\eta-\sigma}{2}}\zeta_K(1+\eta).
\end{align*}
\end{lemma}
\begin{proof}
See \cite[(5.4)]{fryska}.
\end{proof}

\begin{lemma}\label{Prachar1}
Let $f(s)$ be a function regular in the circle $|s-s_0|\leq r$ and satisfying the in equality
\begin{align*}
\left|\frac{f(s)}{f(s_0)}\right|\leq M.
\end{align*}
If $f(s)\neq 0$ in the region $|s-s_0|\leq \frac{r}{2}$, $\Re(s-s_0)>0$, then
\begin{align*}
\Re\frac{f'}{f}(s_0) \geq -\frac{4}{r}\log M.
\end{align*}
\end{lemma}
\begin{proof}
See \cite{Prachar}, p. 384-385
\end{proof}
\begin{lemma}\label{Prachar2}
Let $f(s)$ be a function regular in the circle $|s-s_0|\leq R$ and satisfying the coditions
\begin{align*}
\Re f(s) \leq M \qquad \mbox{for} \qquad |s-s_0|=R
\end{align*}
Then
\begin{align*}
|f^{(k)}(s)|\leq 2k!(M -\Re f(s_0))\frac{R}{(R-r)^{k+1}}, \qquad k\geq 1.
\end{align*}
in the circle $|s-s_0|\leq r < R$.
\end{lemma}
\begin{proof}
See \cite{Prachar}, p. 384-385
\end{proof}

We are in a position to prove Lemma \ref{PochodnaLogarytmicznaZeta}.
\begin{proof}
Let $B = \frac{A_1}{6}=0.01325$, where $A_1$ is the constant appearing in (\ref{obszar_wolny_od_zer}). Let $s_0=\sigma_0+it_0$, $t_0\geq 0$,
\begin{align}\label{stalasigma_zero}
\sigma_0=1+\frac{B}{L(t_0)}.
\end{align}
where $L(t_0)$ is  defined in (\ref{sta쓰L}). We define the function
\begin{align*}
H(s,\chi)=\log\frac{g(s,\chi)}{g(s_0,\chi)}, \quad g(s)=h(s,\chi)\prod_{\rho}(s-\rho)^{-1},
\end{align*}
where $h(s,\chi)=\zeta(s,\chi)$ if $\chi\neq\chi_0$ and $h(s,\chi_0)=(s-1)\zeta(s,\chi_0)$, where  $\rho$ are  zeros of the function $h(s,\chi)$ in the circle $|s-s_0|\leq \frac{1}{2}$.
Firstly, we estimate $\left|\frac{g(s)}{g(s_0)}\right|$. Lemmas \ref{Lematzeta_notprincipal}, \ref{Lematodwrotnosc_zety} and \ref{OszacZetaGlowny} shows that in the circle $|s-\sigma_0|\leq 1$
\begin{align}\label{lematwaski przedzialZeta1}
\left|\frac{\zeta(\sigma+it,\chi)}{\zeta(\sigma_0,\chi)}\right| \leq 1.4(1 +\varepsilon_\chi)A(\mathfrak{f})^{1+2\eta}\zeta(1+\eta)^{2r_2}\zeta_K(\sigma_0) (t_0+2)^{r_2(1+2\eta)},
\end{align}
for any character $\chi\neq \chi_0$ modulo $\mathfrak{f}$, where $\varepsilon_\chi=0$ or 1 to accordingly whether $\chi$ is primitive or not, and
\begin{align}\label{lematwaski przedzialZeta1A}
& \left|\frac{\zeta(\sigma+it,\chi_0)(s-1)}{\zeta(\sigma_0,\chi_0)(\sigma_0-1)}\right|\leq\\
& \qquad \qquad  \leq \frac{L(t_0)}{B}(4+|t_0|)(2+|t_0|)^{r_2(1+\eta)} (|\Delta| N\mathfrak{f})^{\frac{1+\eta}{2}}\zeta_K(1+\eta)\zeta_K(\sigma_0). \nonumber
\end{align}
On $|s-s_0|=1$, $|s_0-\rho| \leq \frac{1}{2}$ and $|s-\rho|\geq \frac{1}{2}$. From (\ref{lematwaski przedzialZeta1}),
(\ref{lematwaski przedzialZeta1A}) and the maximum principle we obtain
\begin{align}\label{oszcowanieilorazuG}
 & \left|\frac{\zeta(s,\chi)\prod_{\rho}(s_0-\rho)}{\zeta(s_0,\chi)\prod_{\rho}(s-\rho)}\right|\leq\\
& \qquad \qquad \leq 1.4(1+\varepsilon_\chi)A(\mathfrak{f})^{1+2\eta} \zeta(1+\eta)^{2r_2}\zeta_K(\sigma_0) (t_0+2)^{r_2(1+2\eta)}, \nonumber
\end{align}
and
\begin{align}\label{oszcowanieilorazuGA}
 & \left|\frac{(s-1)\zeta(s,\chi_0)\prod_{\rho}(s_0-\rho)}{(s_0-1)\zeta(s_0,\chi_0)\prod_{\rho}(s-\rho)}\right|\leq \left|\frac{(s-1)\zeta(s,\chi_0)\prod_{\rho}(s_0-\rho)}{(\sigma_0-1)\zeta(s_0,\chi_0)\prod_{\rho}(s-\rho)}\right|\\
& \qquad \qquad \leq \frac{L(t_0)}{B}(4+|t_0|)(2+|t_0|)^{r_2(1+\eta)} (|\Delta| N\mathfrak{f})^{\frac{1+\eta}{2}}\zeta_K(1+\eta)\zeta_K(\sigma_0) \nonumber
\end{align}
in the circle $|s-s_0|\leq 1$. Secondly, we apply Lemma \ref{Prachar2} to the function $H(s,\chi)$ with $k=1$, $R=\frac{1}{2}$ and $r=\frac{1}{4}$. The function $H(s,\chi)$ is regular in the circle
$|s-s_0|\leq \frac{1}{2}$, so by (\ref{oszcowanieilorazuG}), (\ref{oszcowanieilorazuG}) we obtain
\begin{align*}
& \Re H(s,\chi)  =\log \left| \frac{g(s,\chi)}{g(s_0,\chi)} \right| \leq \\
& \quad \leq \left\{ \begin{array}{lcc}
                   \log\left(1.4(1+\varepsilon_\chi)A(\mathfrak{f})^{1+2\eta} \zeta(1+\eta)^{2r_2}\zeta_K(\sigma_0) (t_0+2)^{r_2(1+2\eta)})\right), & \mbox{if} & \chi\neq\chi_0\\
                    \log \left(\frac{L(t_0)}{B}(4+|t_0|)(2+|t_0|)^{r_2(1+\eta)} (|\Delta| N\mathfrak{f})^{\frac{1+\eta}{2}}\zeta_K(1+\eta)\zeta_K(\sigma_0)\right), & \mbox{if} & \chi=\chi_0.
                     \end{array}\right.
\end{align*}
in the circle $|s-s_0|\leq \frac{1}{2}$. Therefore, in the circle $|s-s_0|\leq\frac{1}{4}$ we have
\begin{align}\label{szcowanie_zer1}
& \left|\frac{\zeta'}{\zeta}(s,\chi) - \sum_{\rho}\frac{1}{s-\rho}\right|  \leq\\
& \qquad \qquad \leq 16\log\left(1.4 (1+\varepsilon_\chi)A(\mathfrak{f})^{1+2\eta} \zeta(1+\eta)^{2r_2}\zeta_K(\sigma_0) (t_0+2)^{r_2(1+2\eta)}\right), \nonumber
\end{align}
and
\begin{align}\label{szcowanie_zer1A}
&\left|\frac{\zeta'}{\zeta}(s,\chi_0) +\frac{1}{s-1} - \sum_{\rho}\frac{1}{s-\rho}\right|  \leq \\
&\qquad \qquad \leq 16 \log\left(\frac{L(t_0)}{B}(4+|t_0|)(2+|t_0|)^{r_2(1+\eta)} (|\Delta| N\mathfrak{f})^{\frac{1+\eta}{2}}\zeta_K(1+\eta)\zeta_K(\sigma_0)\right).  \nonumber
\end{align}
Finally, we estimate $|\sum_{\rho}\frac{1}{s_0-\rho}|$ and $|\sum_{\rho}\frac{1}{s-\rho}|$. In \cite{Israilov}  Israilov  show that, if $1 < \sigma \leq 2$ then
\begin{align*}
-\frac{\zeta'}{\zeta}(\sigma) <\frac{1}{\sigma -1} - \gamma +C_1(\sigma-1),
\end{align*}
where $C_1=0.1875463$ . Hence, (\ref{stalasigma_zero}) shows
\begin{align}\label{szcowanie_zer2}
\left|\frac{\zeta'}{\zeta}(s_0,\chi) \right|  \leq 2r_2 \frac{L(t_0)}{B} + 2r_2C_1\frac{B}{L(t_0)},
\end{align}
and
\begin{align}\label{szcowanie_zer2A}
\left|\frac{\zeta'}{\zeta}(s_0,\chi_0) +\frac{1}{s_0-1}\right|  \leq 4r_2\frac{L(t_0)}{B} + 2r_2C_1\frac{B}{L(t_0)}.
\end{align}
By (\ref{szcowanie_zer1}),   (\ref{szcowanie_zer2}),   we obtain
\begin{align}\label{oszacowanie_pochLog2}
\left|\sum_{\rho}\frac{1}{s_0-\rho}\right|& \leq16\log\left(1.4 (1+\varepsilon_\chi)A(\mathfrak{f})^{1+2\eta} \zeta(1+\eta)^{2r_2}\zeta_K(\sigma_0) (t_0+2)^{r_2(1+2\eta)})\right)\\
& +  2r_2 \frac{L(t_0)}{B} + 2r_2C_1\frac{B}{L(t_0)}, \nonumber
\end{align}
and (\ref{szcowanie_zer1A}), (\ref{szcowanie_zer2A})
\begin{align}\label{oszacowanie_pochLog2A}
\left|\sum_{\rho}\frac{1}{s_0-\rho}\right|  & \leq 16 \log \left(\frac{L(t_0)}{B}(4+|t_0|)(2+|t_0|)^{r_2(1+\eta)} (|\Delta| N\mathfrak{f})^{\frac{1+\eta}{2}}\zeta_K(1+\eta)\zeta_K(\sigma_0)\right)\\
& + 4r_2\frac{L(t_0)}{B} + 2r_2C_1\frac{B}{L(t_0)}. \nonumber
\end{align}
in the circle $|s-s_0|\leq\frac{1}{4}$. Now, we define
\begin{align*}
r_1=\frac{2B}{L(t_0)}< \frac{1}{4}.
\end{align*}
By Theorem \ref{Fryska_obszr_wolnyodZer}, the function $\zeta(s,\chi)\neq 0$ in the region $|s-s_0|\leq r$, $\Re(s-s_0) > -2r_1$. Hence
\begin{align*}
|s_0-\rho|\geq 2r_1, \quad |s-\rho|\geq \frac{1}{2}|s_0-\rho|, \quad \Re(s_0-\rho)\geq 2r_1
\end{align*}
for all zeros $\rho$ in the circle $|s-s_0|\leq \frac{1}{4}$, and for $s$ in the circle $|s-s_0|\leq r_1$.  For $|s-s_0|\leq r_1$  we obtain
\begin{align*}
\left|\sum_{\rho} \frac{1}{s-\rho}- \sum_{\rho} \frac{1}{s_0-\rho}\right| & \leq \sum_{\rho} \frac{|s-s_0|}{|s-\rho||s_0-\rho|}\leq \sum_{\rho} \frac{r_1}{\frac{1}{2}|s_0-\rho|^2}\\
& \leq \sum_{\rho} \frac{\Re(s_0-\rho)}{|s_0-\rho|^2}\ \leq \sum_{\rho} \Re\frac{1}{s_0-\rho} \leq \left|\sum_{\rho} \frac{1}{s_0-\rho} \right|.
\end{align*}
Thus,
\begin{align}\label{oszacowanie_pochLog3}
\left|\sum_{\rho} \frac{1}{s-\rho}\right| & \leq 2 \left|\sum_{\rho} \frac{1}{s_0-\rho} \right|.
\end{align}
From (\ref{szcowanie_zer1}), (\ref{oszacowanie_pochLog2}) and (\ref{oszacowanie_pochLog3}) we have
\begin{align*}
\left|\frac{\zeta'}{\zeta}(s,\chi)\right| & \leq
 4r_2 \frac{L(t_0)}{B} + 4r_2C_1\frac{B}{L(t_0)} \\
&+  32\log\left(1.4 (1+\varepsilon_\chi)A(\mathfrak{f})^{1+2\eta} \zeta(1+\eta)^{2r_2}\zeta_K(\sigma_0) (t_0+2)^{r_2(1+2\eta)}\right),
\end{align*}
and by (\ref{szcowanie_zer1}), (\ref{oszacowanie_pochLog2A}) and (\ref{oszacowanie_pochLog3})
\begin{align*}
\left|\frac{\zeta'}{\zeta}(s,\chi) + \frac{1}{s-1}\right| & \leq 8r_2\frac{L(t_0)}{B} + 4r_2C_1\frac{B}{L(t_0)}\\
& + 32 \log \left(\frac{L(t_0)}{B}(4+|t_0|)(2+|t_0|)^{r_2(1+\eta)} (|\Delta| N\mathfrak{f})^{\frac{1+\eta}{2}}\zeta_K(1+\eta)\zeta_K(\sigma_0)\right)
\end{align*}
in the circle $|s-s_0|\leq r_1$, and consequently in the strip
\begin{align*}
1-\frac{B}{L(t)}=1-\frac{A_1}{6L(t)}< \sigma <1 + \frac{3B}{L(t)}=1 + \frac{A_1}{2L(t)}.
\end{align*}
Put $s_0=\sigma_0+t_0$, $t_0\geq 0$, and
\begin{align*}
 1 + \frac{A_1}{2L(t_0)} \leq \sigma \leq 3.
\end{align*}
Lemma \ref{Prachar1} and (\ref{lematwaski przedzialZeta1}), (\ref{lematwaski przedzialZeta1A}) shows that
\begin{align*}
 -\frac{\zeta '}{\zeta}(\sigma_0,\chi)\leq 4 \log \left(1.4(1 +\varepsilon_\chi)A(\mathfrak{f})^{1+2\eta}\zeta(1+\eta)^{2r_2}\zeta_K(\sigma_0) (t_0+2)^{r_2(1+2\eta)}\right)
\end{align*}
for any character $\chi\neq \chi_0$ modulo $\mathfrak{f}$, where $\varepsilon_\chi=0$ or 1 to accordingly whether $\chi$ is primitive or not, and
\begin{align*}
 -\frac{\zeta '}{\zeta}(\sigma_0,\chi_0) & \leq \frac{1}{\sigma_0 -1} + \\
 & + 4 \log \left(\frac{L(t_0)}{B}(4+|t_0|)(2+|t_0|)^{r_2(1+\eta)} (|\Delta| N\mathfrak{f})^{\frac{1+\eta}{2}}\zeta_K(1+\eta)\zeta_K(\sigma_0)\right).
\end{align*}
Therefore,
\begin{align*}
 \left|\frac{\zeta '}{\zeta}(\sigma_0,\chi)\right|\leq 4 \log \left(1.4(1 +\varepsilon_\chi)A(\mathfrak{f})^{1+2\eta}\zeta(1+\eta)^{2r_2}\zeta_K(\sigma_0) (t_0+2)^{r_2(1+2\eta)}\right)
\end{align*}
for any character $\chi\neq \chi_0$ modulo $\mathfrak{f}$, where $\varepsilon_\chi=0$ or 1 to accordingly whether $\chi$ is primitive or not, and
\begin{align*}
 \left|\frac{\zeta '}{\zeta}(\sigma_0,\chi_0)\right|& \leq \frac{2L(t_0)}{A_1} + \\
 & + 4 \log \left(\frac{L(t_0)}{B}(4+|t_0|)(2+|t_0|)^{r_2(1+\eta)} (|\Delta| N\mathfrak{f})^{\frac{1+\eta}{2}}\zeta_K(1+\eta)\zeta_K(\sigma_0)\right).
\end{align*}
By the above we obtain
\begin{align*}
 \left|\frac{\zeta '}{\zeta}(s_0,\chi)\right| & \leq \left|\frac{\zeta '}{\zeta}(\sigma_0,\chi)\right| \leq\\
     & 4 \log \left(1.4(1 +\varepsilon_\chi)A(\mathfrak{f})^{1+2\eta}\zeta(1+\eta)^{2r_2}\zeta_K(\sigma_0) (t_0+2)^{r_2(1+2\eta)}\right)
\end{align*}
for any character $\chi\neq \chi_0$ modulo $\mathfrak{f}$, where $\varepsilon_\chi=0$ or 1 to accordingly whether $\chi$ is primitive or not, and
\begin{align*}
 &\left|\frac{\zeta '}{\zeta}(s_0,\chi_0)  + \frac{1}{s_0-1}\right| \leq \left|\frac{\zeta '}{\zeta}(\sigma_0,\chi)\right| + \frac{1}{\sigma_0-1}\leq   \frac{4L(t_0)}{A_1} + \\
 & \qquad \qquad+ 4 \log \left(\frac{L(t_0)}{B}(4+|t_0|)(2+|t_0|)^{r_2(1+\eta)} (|\Delta| N\mathfrak{f})^{\frac{1+\eta}{2}}\zeta_K(1+\eta)\zeta_K(\sigma_0)\right).
\end{align*}
The proof is completed by applying
\begin{align*}
\zeta_K(\sigma_0)\leq\zeta(\sigma_0)^{2r_2}\leq \left(1+\frac{6L(t_0)}{A_1}\right)^{2r_2}.
\end{align*}
\end{proof}

\bibliographystyle{splncs03}
\bibliography{citations}

\end{document}